\documentclass[12pt,reqno, oneside]{amsart}
\usepackage[height=8.5in, width=6.5in]{geometry}
\usepackage{hyperref}
\usepackage{mathrsfs}
\newcommand{\seqnum}[1]{\href{http://oeis.org/#1}{{#1}}}
\newcommand{\ud}{up-up-or-down-down}
\newcommand{\sq}{\frac{1}{\sqrt2}}
\newcommand{\hf}{\frac12}
\newcommand{\2}{\sqrt{2}}
\renewcommand{\L}{\mathscr{L}}

\newtheorem{thm}{Theorem}
\newtheorem{lem}[thm]{Lemma}
\newtheorem*{theorem}{Theorem}
\newcommand{\abs}[1]{\left|#1\right|}
\newcommand{\x}[1]{\frac{x^{#1}}{#1!}}
\newcommand{\xp}[1]{\frac{x^{#1}}{(#1)!}}
\newcommand{\y}[1]{\frac{y^{#1}}{#1!}}
\newcommand{\yp}[1]{\frac{y^{#1}}{(#1)!}}

\begin{document}
\title{Counting up-up-or-down-down permutations}
\author{Ira M. Gessel}
\address{Department of Mathematics\\
   Brandeis University\\
   Waltham, MA 02453}
\email{gessel@brandeis.edu}
\date{\today}
\begin{abstract}
Answering a question of Donald Knuth, we find the bivariate exponential generating function for ``up-up-or-down-down'' permutations  of odd length according to their last entry. An \ud\ permutation is a permutation $a_1a_2\cdots a_n$ satisfying $a_{2i-1}<a_{2i}$ if and only if $a_{2i}<a_{2i+1}$ for $1\le i <n/2$. Equivalently, an \ud\ permutation is one in which every peak and every valley is odd.

Let  $p_n(k)$, for $-n\le k \le n$,  be the number of \ud\ permutations of the $(2n+1)$-element set $\{-n,-n+1,\dots, 0, \dots, n-1, n\}$ that end with $k$. Our starting point for the derivation of the generating function is the recurrence 
$
p_{n+1}(k)=\sum_{j=-n}^n \abs{j-k} p_n(j)
$
found by Knuth.

\end{abstract}

\maketitle
\thispagestyle{empty}

\section{Introduction}
Donald Knuth \cite{whirlpool} introduced \emph{whirlpool permutations}, which he defines  in the following way: An $m \times n$ matrix has $(m-1)(n-1)$ submatrices  of size $2 \times 2$ consisting of adjacent rows and columns. An $m \times n$ whirlpool permutation is an $m \times n$ matrix with entries the integers from 1 to  $mn$, in which the relative order of the elements in each of those submatrices is a ``vortex''---that is, it travels a cyclic path from smallest to largest, either clockwise or counterclockwise.

Knuth computed the number of $m\times n$ whirlpool permutations for small values of $m$ and~$n$ and found that the number $W_n$ of $2\times n$ whirlpool permutations seemed be given by sequence \seqnum{A261683} of the On-Line Encyclopedia of Integer Sequences (OEIS) \cite{oeis}, which counts permutations $a_1a_2\cdots a_{2n}$ satisfying $a_{2i-1}<a_{2i}$ if and only if $a_{2i}<a_{2i+1}$ for $1\le i\le n-1$.  He proved that this is indeed the case. For example, if $n=2$ then there are eight such permutations and they satisfy either $a_1<a_2<a_3$ or $a_1>a_2>a_3$. Note that there is no restriction on $a_{2n}$. Thus $W_n = 2n U_n$, where $U_n$ is the number of permutations of $\{1,2,\dots, 2n-1\}$ satisfying $a_{2i-1}<a_{2i}$ if and only if $a_{2i}<a_{2i+1}$ for $1\le i\le n-1$. 

Knuth called permutations satisfying $a_{2i-1}<a_{2i}$ if and only if $a_{2i}<a_{2i+1}$ \emph{\ud\ permutations} since the ups and downs (often called \emph{rises and falls} or \emph{ascents and descents}) come in pairs. They may be described in a slightly different way: For a permutation $a_1a_2\cdots a_m$, an index $i$, with $1<i< m$, is called a \emph{peak} if $a_{i-1}<a_i>a_{i+1}$, and $i$ is called a \emph{valley} if $a_{i-1}> a_i < a_i$. Then the \ud\ permutations are those in which every peak and every valley is odd.

The numbers $U_n$ that count \ud\ permutations of odd length are sequence \seqnum{A122647} in the OEIS. They have the generating function
\begin{align}
\sum_{n=0}^\infty U_n \frac{x^{2n+1}}{(2n+1)!}
   &=\frac{\sqrt2\tanh(x/\sqrt2)}{1-(x/\sqrt2)\tanh(x/\sqrt2)}=\frac{\2\sinh(x/\2)}{\cosh(x/\2) -(x/\2)\sinh(x/\2)},\label{e-0}\\
   &=x+2\x{3}+14\x{5}+204\x{7}+5104\x{9}+10570416\x{13} +\cdots, \notag
\end{align}
as shown by La Croix \cite{lacroix}, Zhuang \cite{zhuang}, and Knuth \cite[Exercise 33(b)]{knuthdraft}. A derivation of the   generating function for $W_n$, which is $x$ times \eqref{e-0}, was given by Basset \cite{basset}.

In his investigation of \ud\ permutations Knuth considered the refined problem of counting \ud\ permutations with a fixed last entry. He defined the number $p_n(k)$, for $-n\le k \le n$, to be the number of \ud\ permutations of the $(2n+1)$-element set $\{-n,-n+1,\dots, 0, \dots, n-1, n\}$ that end with $k$. Then Knuth (personal communication) observed that the numbers $p_n(k)$ for $-n\le k\le n$ are determined by the recurrence
\begin{equation}
\label{e-1}
p_{n+1}(k)=\sum_{j=-n}^n \abs{j-k} p_n(j)
\end{equation}
for $n\ge0$,
with the initial value $p_0(0)=1$, and asked if a generating function for $p_n(k)$ could be obtained from \eqref{e-1}.  In this paper we show how this can be done. 

Here are the values of $p_n(k)$ for small $n$. 
\[
\offinterlineskip\vbox
{\halign{\ \hfil\strut$#$\hfil\ \vrule&&\hfil\quad $#$\quad\cr
n\backslash k&-4&-3&-2&-1&0&1&2&3&4\cr
\noalign{\hrule}
&&&&&1&&&&\cr
1&&&&1&0&1&&&\cr
2&&&4&2&2&2&4&&\cr
3&&42&28&22&20&22&28&42&\cr
4&816&612&492&428&408&428&492&612&816\cr
}}
\]

\section{The Proof}

First we verify that Knuth's recurrence \eqref{e-1} holds.

\begin{lem}
\label{l-red}
The numbers $p_n(k)$ satisfy the recurrence \eqref{e-1}.
\end{lem}
\begin{proof}
Given an \ud\ permutation $a_1a_2\cdots a_{2n+3}$ of $\{-(n+1), -n,\dots, n, n+1\}$ with $a_{2n+3}=k$, we remove $a_{2n+2}$ and $a_{2n+3}$ and replace the remaining entries in an order-preserving way with $-n,\dots, n$, thus obtaining an \ud\ permutation $b_1b_2\cdots b_{2n+1}$ of $\{-n,\dots,n\}$. How many such permutations $a_1a_2\cdots a_{2n+3}$ are there with $b_{2n+1}=j$?
We consider two cases. 

First, suppose that $a_{2n+1}<a_{2n+2}<a_{2n+3}=k$. In this case, $b_{2n+1}=a_{2n+1}+1$, so $a_{2n+1}=b_{2n+1}-1=j-1$ and $j<k$.  Then $a_{2n+2}$ may be any of the $k-j$ numbers $j, j+1,\dots, k-1$, and $a_1a_2\cdots a_{2n+3}$ is uniquely determined by  $b_1b_2\cdots b_{2n+1}$ and the choice of $a_{2n+2}$, so there are $(k-j)p_n(j)$ possibilities for  $a_1a_2\cdots a_{2n+3}$. 

In the second case, suppose that $a_{2n+1}>a_{2n+2}>a_{2n+3}$. In this case $b_{2n+1}=a_{2n+1}-1$, so $a_{2n+1}=b_{2n+1}+1=j+1$ and $j>k$. Then $a_{2n+2}$ may be any of the $j-k$ numbers $k+1, k+2,\dots, j$, so as in the first case, there are $(j-k)p_{n}(j)$ possibilities for  $a_1a_2\cdots a_{2n+3}$. 

Thus 
\begin{equation*}
p_{n+1}(k)=\sum_{j=-n}^{k-1} (k-j)p_n(j) + \sum_{j=k+1}^n (j-k) p_n(j) = \sum_{j=-n}^n |j-k| p_n(j).\qedhere
\end{equation*}

\end{proof}

Next, we derive from \eqref{e-1} three identities for $p_n(k)$ that are easier to work with then \eqref{e-1}.

\begin{lem}
\label{l-0}
For $n\ge0$ and $-n\le k\le n$ we have  $p_n(-k)=p_n(k)$.
\end{lem}
\begin{proof}
This follows easily by induction from \eqref{e-1}.
\end{proof}

\begin{lem}
\label{l-1}
For $n\ge0$ and $-n\le k\le n$ we have
$p_{n+1}(k+1)-2p_{n+1}(k)+p_{n+1}(k-1) =2p_{n}(k)$.
\end{lem}
\begin{proof}
By \eqref{e-1}, we have
\begin{multline*}
\qquad\quad
p_{n+1}(k+1)-2p_{n+1}(k)+p_{n+1}(k-1) \\
   =\sum_{j=-n}^n \bigl(\abs{j-k-1}-2\abs{j-k}+\abs{j-k+1} \bigr)p_n(j).
\qquad\quad
\end{multline*}
If $j>k$ then 
\begin{equation*}
\abs{j-k-1}-2\abs{j-k}+\abs{j-k+1} = (j-k-1) -2(j-k) +(j-k+1) = 0.
\end{equation*}
If $j<k$ then this sum is similarly 0. Finally, if $j=k$ then 
\begin{equation*}
\abs{j-k-1}-2\abs{j-k}+\abs{j-k+1} = 1-2\cdot0+1=2.\qedhere
\end{equation*}
\end{proof}

\begin{lem}
\label{l-2}
For  $n\ge1$ we have 
$(n-1)p_n(n) = np_n(n-1)$.
\end{lem}
\begin{proof}
We prove the equivalent formula $np_{n+1}(n+1)=(n+1)p_{n+1}(n)$ for $n\ge0$.
Using the fact that $p_n(-j)=p_n(j)$, we have 
\begin{align*}
p_{n+1}(n+1)&=\sum_{j=-n}^n \abs{j-n-1}p_n(j)\\
  &=(n+1)p_n(0) +\sum_{j=-n}^{-1}\abs{j-n-1}p_n(j) +\sum_{j=1}^{n}\abs{j-n-1}p_n(j)\\
  &=(n+1)p_n(0)+\sum_{j=1}^{n}(n+j+1)p_n(j) +\sum_{j=1}^{n}(n-j+1)p_n(j)\\
  &=(n+1)p_n(0)+2(n+1)\sum_{j=1}^{n}p_n(j).
\end{align*}
Similarly, we have
\begin{align*}
p_{n+1}(n)&=\sum_{j=-n}^n \abs{j-n}p_n(j)\\
 &=n p_n(0) +\sum_{j=-n}^{-1}\abs{j-n}p_n(j) +\sum_{j=1}^{n}\abs{j-n}p_n(j)\\
  &=n p_n(0)+\sum_{j=1}^{n}(n+j)p_n(j) +\sum_{j=1}^{n}(n-j)p_n(j)\\
  &=n p_n(0)+2n\sum_{j=1}^{n}p_n(j),
\end{align*}
and the equality $np_{n+1}(n+1)=(n+1)p_{n+1}(n)$ follows.
\end{proof}

Next we need two facts about ``Seidel arrays." (Cf.~Dumont \cite{dumont} and Dumont and Viennot \cite{dv}.) 

\begin{lem}
\label{l-3}
\textup{(a)}
Let $a_{i,j}$ be an array of numbers, indexed by nonnegative integers $i$ and $j$, and satisfying $a_{i+1,j}-a_{i,j+1}=2 a_{i,j}$ for all nonnegative integers $i$ and $j$.
Then \begin{equation*}
\sum_{i,j=0}^\infty a_{i,j}\x i \y j = e^{-2 y}A(x+y),
\end{equation*}
where $A(x) = \sum_{i=0}^\infty a_{i,0} x^i/i!$.

\textup{(b)} 
Now suppose that in addition $a_{j,i}=(-1)^{i+j}a_{i,j}$ for all $i$ and $j$. Then
\begin{equation*}
\sum_{i,j=0}^\infty a_{i,j}\x i \y j = e^{  (x-y)}B(x+y),
\end{equation*}
where $B(x)$ is a power series with only even powers of $x$,
and thus
\begin{equation*}
\sum_{\textup{$i$+$j$ even}} a_{i,j}\x i \y j 
  =\cosh(x-y)\, B(x+y)
\end{equation*}
and 
\begin{equation*}
\sum_{\textup{$i$+$j$ odd}} a_{i,j}\x i \y j 
  =\sinh(x-y)\, B(x+y).
\end{equation*}

\end{lem}
\begin{proof}
Let\\[-20pt]
\[R(x,y) = \sum_{i,j=0}^\infty a_{i,j}\x i \y j.\]
Then the recurrence  $a_{i+1,j}-a_{i,j+1}=2 a_{i,j}$ is equivalent to the differential equation
\begin{equation}
\label{e-R}
\frac{\partial R}{\partial x}(x,y) - \frac{\partial R}{\partial y}(x,y) = 2 R(x,y).
\end{equation}
Since the entire array $(a_{i,j})$ is determined by the recurrence and the values of $a_{i,0}$, it follows that $R(x,y)$ is the unique formal power series solution of \eqref{e-R} satisfying $R(x,0)= A(x)$. Since $e^{-2 y}A(x+y)$ satisfies the differential equation and reduces to $A(x)$ for $y=0$, it must be equal to $R(x,y)$.

The condition $a_{j,i}=(-1)^{i+j}a_{i,j}$ is equivalent to $R(y,x) = R(-x,-y)$. If this holds then $e^{-2 x}A(x+y) = e^{2 y}A(-x-y)$. Setting $y=0$ gives $e^{-2 x}A(x) = A(-x)$ so $B(x) =e^{-  x}A(x)$ satisfies $B(x) = B(-x)$, and $e^{-2 y}A(x+y)= e^{(x-y)}B(x+y)$.
\end{proof}

We note that in part (b) of Lemma \ref{l-3}, instead of the hypothesis $a_{j,i}=(-1)^{i+j}a_{i,j}$, it would be sufficient to have $a_{0,i}=(-1)^i a_{i,0}$.

We can now prove the main result.

\begin{theorem}
\begin{equation}
\label{e-2var}
\sum_{n=0}^\infty \sum_{k=-n}^n  p_n(k) \xp{n+k}\yp{n-k}
  = \frac{\cosh\Bigl(\sq(x-y)\Bigr)}{\cosh\Bigl(\sq(x+y)\Bigr) -\sq(x+y)\sinh\Bigl(\sq(x+y)\Bigr)}.
\end{equation}
\end{theorem}
\begin{proof}
Let 
 \[P=P(x,y) = \sum_{n=0}^\infty \sum_{k=-n}^n 2^n p_n(k) \xp{n+k}\yp{n-k}.\]
 (The power of two is inserted to avoid factors of $\2$ in many of our formulas.) 
 By Lemma~\ref{l-0}, $P(y,x)=P(x,y)$, and since every term in $P(x,y)$ has even total degree, we have $P(-y,-x)=P(x,y)$.

Let $\L$ be the linear operator on power series in $x$ and $y$ defined by 
\begin{equation*}
\L(G) = \frac{\partial G}{\partial x} - \frac{\partial G}{\partial y}.
\end{equation*}
Then the recurrence of Lemma \ref{l-1} is equivalent to 
\begin{math}
\L^2(P) = 4P.
\end{math}
Now let $Q=Q(x,y) = \hf\L(P)$
and let $R=R(x,y)=P+Q$. Then \[\L(R)=\L(P)+\tfrac12\L^2(P)=2 Q+2P = 2R.\]

It is easy to check that since $P(-y,-x) = P(x,y)$, we have $Q(-y,-x) = Q(x,y)$. 
Thus by part (b) of Lemma \ref{l-3}, we have
\begin{equation}
\label{e-RB}
R= e^{x-y}B(x+y),
\end{equation}
where $B(x)$ is a power series with only even powers of $x$. All we need to do now is determine $B(x)$.
From \eqref{e-RB} we have
\begin{equation}
\label{e-PQB}
\begin{aligned}
P(x,0) &= \cosh x\, B(x),\\
Q(x,0) &= \sinh x\,  B(x).
\end{aligned}
\end{equation}

Next, we want to show that 
$P(x,0)=1+xQ(x,0)$.
We have 
\[P(x,0) = \sum_{n=0}^\infty 2^n p_n(n) \xp{2n}\]
and it follows from the definition of $Q(x,y)$ that
\[Q(x,0) =  \sum_{n=1}^\infty 2^{n-1}\bigl(p_n(n) -p_n(n-1)\bigr) \xp{2n-1},\]
so
\begin{align*}
x Q(x,0) &= \sum_{n=1}^\infty x2^{n-1}\bigl(p_n(n) -p_n(n-1)\bigr) \xp{2n-1}\\
  &=\sum_{n=1}^\infty 2^n n\bigl(p_n(n) -p_n(n-1)\bigr) \xp{2n}\\
  &=\sum_{n=1}^\infty 2^n p_n(n) \xp{2n}\quad \text{(by Lemma \ref{l-2})}\\
  &=P(x,0)-1.
\end{align*}

Thus from \eqref{e-PQB} we obtain
\begin{equation}
\label{e-PQ1}
\begin{aligned}
P(x,0) &= \cosh x\, B(x)\\
P(x,0) &=1+x\sinh x\, B(x).
\end{aligned}
\end{equation}
Solving for $B(x)$ in \eqref{e-PQ1} gives
\[B(x) = \frac{1}{\cosh x - x\sinh x}.\]
Thus
\[R(x,y) = e^{x-y}B(x+y) = \frac{e^{x-y}}{\cosh(x+y)-(x+y)\sinh(x+y)}\]
and $P(x,y)$ is the sum of the terms in $R(x,y)$ of even total degree, so 
\begin{equation*}
P(x,y) = \frac{\cosh(x-y)}{\cosh(x+y)-(x+y)\sinh(x+y)}.
\end{equation*}

Then \eqref{e-2var} follows by replacing $x$ with $x/\sqrt2$ and $y$ with $y/\sqrt2$.
\end{proof}

\section{Further remarks}

We can derive the one-variable generating function \eqref{e-0} from \eqref{e-2var} in two different ways. First, we may set $y=0$ in \eqref{e-2var}, obtaining
\begin{equation}
\label{e-1var}
\sum_{n=0}^\infty  p_n(n) \xp{2n}
  = \frac{\cosh(x/\sqrt2)}{\cosh(x/\sqrt2) -(x/\sqrt2)\sinh(x/\sqrt2)}.
\end{equation}
Now $p_n(n)$ is the number of \ud\ permutations of $\{-n,-n+1,\dots, 0, \dots n-1, n\}$ that end with $n$. If we remove the final $n$, then for $n>0$ we have an \ud\ permutation of $2n$ numbers ending with an up. Therefore if we subtract 1 from \eqref{e-1var} and  multiply by 2, we will obtaining the exponential generating function for the numbers $W_n=2nU_n$ as defined in the introduction. Dividing by $x$ gives \eqref{e-0}.

With a little more work, we can derive \eqref{e-0}  by using \eqref{e-2var} to evaluate $\sum_{k=-n}^n p_n(k)$. We replace $x$ with $tx$ and $y$ with $(1-t)x$ in \eqref{e-2var} and integrate with respect to $t$ from 0 to 1, using the beta integral
\begin{equation*}
\int_0^1 t^{n+k}(1-t)^{n-k} = \frac{(n+k)!\,(n-k)!}{(2n+1)!}.
\end{equation*}
We then multiply by $x$. The result is \eqref{e-0}.

One might wonder why the generating function for $p_n(k)$ is an exponential generating function in two variables, rather than an exponential generating function in one variable for the Laurent polynomials  $\sum_{k=-n}^n p_n(k)t^k$. One way of combinatorially explaining the bivariate exponential generating function, which arises from a different approach and from other problems involving permutations where the last entry is fixed, is to restate the problem as counting permutations of a set $\{-m, -m-1, \dots, 0, \dots, n-1, n\}$ that end with~0. In these problems it is usually best to think of the positive and negative numbers as different sorts of labels, and of the final 0 as not a label at all, so the number of permutations in question of such a set will be multiplied by $\frac{x^m}{m!}\frac{y^n}{n!}$. 

We give two examples. First, the bivariate generating function for alternating permutations (satisfying $a_1>a_2<a_3>\dots$) is 
\begin{equation}
\label{e-alt}
\frac{\cos y + \sin y}{\cos(x + y)}.
\end{equation}
In other words, the coefficient of  $\frac{x^m}{m!}\frac{y^n}{n!}$ in \eqref{e-alt} is the number of alternating permutations of the set $\{-m,\dots, 0, \dots, n\}$ that end with 0. The coefficients here are sequence \seqnum{A008280} in the OEIS. They are usually called Entringer numbers, though they appear (without a combinatorial interpretation) in an 1877 paper of Seidel.

Second, let $A_{m,n}(t)$ count by descents permutations $\{-m, -m+1,\dots,0,\dots, n\}$ that end with 0. (So $A_{m,0}(t)$ and $A_{0,n}(t)$ are Eulerian polynomials.)
Then 
\begin{equation*}
\sum_{m,n=0}^\infty \frac{A_{m,n}(t)}{(1-t)^{m+n+1}}\frac{x^m}{m!}\frac{y^n}{n!} = \frac{e^x}{1-te^{x+y}}.
\end{equation*}

\bibliography{up-up}{}

\providecommand{\bysame}{\leavevmode\hbox to3em{\hrulefill}\thinspace}
\begin{thebibliography}{1}

\bibitem{basset}
Nicolas Basset, \emph{Counting and generating permutations in regular classes},
  Algorithmica \textbf{76} (2016), no.~4, 989--1034.

\bibitem{dumont}
Dominique Dumont, \emph{Matrices d'{E}uler-{S}eidel}, S\'eminaire Lotharingien
  de Combinatoire \textbf{5} (1981), B05c, 16 pp.

\bibitem{dv}
Dominique Dumont and G\'erard Viennot, \emph{A combinatorial interpretation of
  the {S}eidel generation of {G}enocchi numbers}, Ann. Discrete Math.
  \textbf{6} (1980), 77--87.

\bibitem{whirlpool}
Don Knuth, \emph{Whirlpool permutations},
  \url{https://www-cs-faculty.stanford.edu/~knuth/papers/whirlpool.pdf}, 2020.

\bibitem{knuthdraft}
Donald~E. Knuth, \emph{The {A}rt of {C}omputer {P}rogramming, {V}olume 4,
  {P}re-{F}ascicle 7a. {A} {D}raft of {S}ection 7.2.2.3: {C}onstraint
  {S}atisfaction}, Addison-Wesley, 2024.

\bibitem{lacroix}
Michael La~Croix, \emph{A combinatorial proof of a result of {G}essel and
  {G}reene}, Discrete Math. \textbf{306} (2006), no.~18, 2251--2256.

\bibitem{oeis}
N.~J.~A. Sloane (ed.), \emph{The {O}n-{L}ine {E}ncyclopedia of {I}nteger
  {S}equences}, published electronically at \url{https://oeis.org}, 2024.

\bibitem{zhuang}
Yan Zhuang, \emph{Counting permutations by runs}, J. Combin. Theory Ser. A
  \textbf{142} (2016), 147--176.

\end{thebibliography}
\bibliographystyle{amsplain}

\end{document}